\documentclass[11pt, leqno]{article}

\usepackage{amsmath,amssymb,amsthm, epsfig}

\usepackage[numbers]{natbib}
\usepackage{hyperref}
\usepackage{amsmath}
\theoremstyle{plain}
\theoremstyle{definition}
\allowdisplaybreaks
\usepackage[pdftex]{color}
\usepackage{comment}

\usepackage{mathtools}
\usepackage{color}

\usepackage{ulem}

\usepackage{mathrsfs}

\usepackage{wasysym}

\usepackage{stackrel}

\usepackage{amsthm} 

\newtheorem{assumptionBase}{\hspace{-5pt}} 
 
\newenvironment{assumption}[1][]
  {\begin{assumptionBase}\normalfont\textbf{\ifx#1\empty\else #1\fi} \normalfont}
  {\end{assumptionBase}}

\usepackage{pgfplots}
\pgfplotsset{compat=newest}

\title{\large{\bf Regularity Estimates for Fully Nonlinear Dead-Core Problems with a Hamiltonian Term}}
\author{\it by \smallskip \\ Rafael R. Costa \footnote{\noindent Universidade Federal do Rio Grande do Sul, Porto Alegre-RS, Brazil \noindent \texttt{E-mail address: \url{ramos.costa@ufrgs.br}}} \quad and\quad Ginaldo S. S\'{a}\footnote{\noindent Universidade Estadual de Campinas - UNICAMP. Departamento  de Matemática. Campinas - SP, Brazil. \noindent \texttt{E-mail address: \url{ginaldo@ime.unicamp.br}}}
}

\newlength{\hchng}
\newlength{\vchng}
\def \R {\mathbb{R}}

\def \M {\text{Sym}}

\def \T {\text{Tr}}

\newtheorem{theorem}{Theorem}[section]

\newtheorem{lemma}[theorem]{Lemma}

\newtheorem{corollary}[theorem]{Corollary}

\theoremstyle{definition}

\newtheorem{definition}[theorem]{Definition}

\newtheorem{example}[theorem]{Example}

\theoremstyle{remark}

\newtheorem{remark}[theorem]{Remark}

\numberwithin{equation}{section}

\newcommand{\intav}[1]{\mathchoice {\mathop{\vrule width 6pt height 3 pt depth  -2.5pt
\kern -8pt \intop}\nolimits_{\kern -6pt#1}} {\mathop{\vrule width
5pt height 3  pt depth -2.6pt \kern -6pt \intop}\nolimits_{#1}}
{\mathop{\vrule width 5pt height 3 pt depth -2.6pt \kern -6pt
\intop}\nolimits_{#1}} {\mathop{\vrule width 5pt height 3 pt depth
-2.6pt \kern -6pt \intop}\nolimits_{#1}}}

\begin{document}
\maketitle

\begin{abstract}
In this paper, we present a problem involving fully nonlinear elliptic operators with Hamiltonian, which can present a singularity or degenerate as the gradient approaches the origin. The model studied here, allows the appearance of plateau zones, i.e. unknown regions of the domain in which the non-negative solutions vanishes. We show an improvement in regularity along the free boundary of the problem, and with some hypotheses
on the exponents of the equation we proved the optimality of the growth rate with the help of the non-degeneracy also obtained here. In addition, some more applications of the growth results on the free boundary are obtained, such as: the positivity of solutions and also information on the Hausdorff measure of the boundary of the coincidence set. 
\medskip
\vspace{0.2cm}

\vspace{0.2cm}
	
\end{abstract}

\newpage

\section{Introduction}

In this work, we investigate geometric properties of the viscosity solution to diffusion-reaction problems ruled by a uniformly elliptic operator with a Hamiltonian term, which may exhibit degenerate/singular behavior when the gradient vanishes, along with a strong absorption term:
\begin{equation}\label{equation}
|\nabla{u}|^pF(D^2u)+\mathfrak{a}(x)|\nabla u|^q=f(x,u)\lesssim\lambda_0(x)u(x)^\mu_+\quad\text{in} \quad \Omega,
\end{equation}
where $\Omega\subset \mathbb{R}^n$ is a smooth, open and bounded domain, $\mathfrak{a},\lambda_0\in \mathrm{C}^0(\overline{\Omega})$ with $\inf\lambda_0> 0$, $F:\M(n)\rightarrow \mathbb{R}$ is a uniformly elliptic operator, where $\M(n)$ denote the space of symmetric matrices of size $n\times n$ and
\begin{itemize}
\item $-1<p$ indicates the degree of singularity/degeneracy of the problem;
\item $0\leq q<p+1$ sublinear exponent;
\item $0\leq \mu<p+1$ represents the strong absorption term of the model. 
\end{itemize}
Additionally, we impose the boundary condition
\begin{equation}\label{Bound-Cond}
 u(x) = g(x) \quad \text{on } \partial \Omega,
\end{equation}
where the boundary datum is assumed to be continuous and non-negative, i.e., \( 0 \leq g \in \mathrm{C}^0(\partial \Omega) \).

The equation approach \eqref{equation} is classified as an equation with a strong absorption term, and the set $\partial\{u>0\}\cap \Omega$ constitutes the free boundary of the problem. The study of such problems has an important role in the construction of abstract models that describe, for example, diffusion phenomena, where $u$ represents the density of a chemical reagent or gas, while $\lambda_0$ controls the ratio between the reaction rate and the diffusion rate. An illustrative example of an equation modeling the density of the reactant $u$ in a stationary situation, see \cite{ADD1,Aris1,Aris2} for some motivations, is given by
$$
\left\{
\begin{array}{rccccc}
-\Delta u+\lambda_0f(u)\chi_{\{u>0\}} &=& 0 & \text{in} & \Omega,\\
u(x) & = & 1& \text{on} & \partial\Omega,
\end{array}
\right.
$$
where $\Omega\subset\R^n$ is a regular and bounded domain and $f$ is a continuous and increasing reaction term, with $f(0)=0$. In this context, $f(u)$ represents the ratio of the reaction rate at concentration $u$ to the reaction rate at concentration unity. From a mathematical perspective, these models are interesting due to the existence of a region where $u\equiv 0$, called the \textit{dead-core}. In the applied context, the \textit{dead-core} corresponds to regions where the density of the substance (or gas) vanishes.

Problems involving dead-core phenomena have been the subject of intense investigation in recent years. Results concerning the existence of solutions, estimates, and geometric properties, among other aspects, have been extensively developed (cf. \cite{BV-P, DH84, DH}, just to cite a few). Another interesting aspect of the mathematical analysis of dead-core problems is its connection with various free boundary problems, see the seminal works \cite{AltPhillips,Diaz1,FP,Phillips} on problems with variational structure. See also \cite{Leitao,ET} for the non-variational counterpart. 

The equation (1.1) has a singular/degenerate structure along the set
\[
\mathcal{C}(u)\coloneqq \{x\in \Omega:\, |\nabla u|=0\}.
\]
Such structure makes it difficult to obtain uniform regularity estimates, especially in neighborhoods of the set \( \mathcal{C}(u)\), where the ellipticity of the operator deteriorates.

In \cite{IS}, the authors proved interior Hölder regularity of the gradient for viscosity solutions in the following setting:
\begin{equation}
|\nabla u|^\gamma F(D^2u) =f(x)\in L^\infty(\Omega),
\end{equation}
where \( \gamma\geq 0 \). In \cite{DRT01}, the authors obtained an optimal \( \mathrm{C}^{1,\alpha} \)-type estimate for the solutions of equation (1.1) along the set of critical points.

In a sequence of celebrated works, Birindelli and Demengel addressed variants of singular/degenerate problems. In \cite{BirDem06}, the authors considered singular/degenerate equations of the form
\begin{equation}\label{BD_1}
|\nabla u|^\gamma F(D^2u) =f(x,\nabla u),
\end{equation}
where \( F \) is fully nonlinear uniformly elliptic, \( \gamma > -1 \), and \( f \) has, for the most part, growth of order \( 1+\gamma \) in the gradient term. In this work, they proved Lipschitz regularity results for viscosity solutions of (1.1). Furthermore, in \cite{BirDem10}, the authors proved regularity $\mathrm{C}^{1,\alpha}$ up to the boundary in the case $\gamma\leq 0$, for the Dirichlet problem with homogeneous boundary conditions.

In the context of fully nonlinear models of singular/degenerate type with Hamiltonian term, Birindelli and Demengel, in \cite{BirDem16}, addressed the boundary regularity of viscosity solutions to the following problem
\begin{equation}
    |\nabla u|^p F(D^2 u)+ |\nabla u|^q b(x) = f(x)\quad\text{in}\quad \Omega,
\end{equation}
where $\Omega\subset \mathbb{R}^n$ 
is a bounded domain of class $\mathrm{C}^2$, $p>-1$, $q\in [0,p+1]$, $b,f\in \mathrm{C}(\overline{\Omega})$ and $F$ is a uniformly elliptic operator. In this context, the authors proved that given a boundary datum $\varphi\in\mathrm{C}^{1,\gamma_0}(\partial\Omega)$, there exist constants $\gamma\in (0,\gamma_0)$, with $\gamma=\gamma(\lambda,\Lambda,p,q)$ and $\mathrm{C}=\mathrm{C}(\gamma)$ such that any bounded viscosity solution $u$ of
$$
\left\{\begin{array}{rrlcc}
   |\nabla u|^pF(D^2u)+|\nabla u|^q b(x) & = & f(x) & \text{in} & \Omega,\\
    u & = & \varphi  & \text{in} & \partial\Omega,
\end{array}
\right.
$$
is $\mathrm{C}^{1,\gamma}$ and for $p<q+1$
$$
\|u\|_{\mathrm{C}^{1,\gamma}(\overline{\Omega})}\leq \mathrm{C}\left(\|u\|_\infty+\|b\|_\infty^{\frac{1}{1+p-q}}+\|f\|_{\infty}^{\frac{1}{1+p}}+\|\varphi\|_{\mathrm{C}^{1,\gamma_0}(\partial\Omega)}\right).
$$
Moreover, for $p=q+1$ (linear case), the exponent $\gamma$ also depends on $\|b\|_\infty$.

We also emphasize the contributions of works \cite{BirDem04, BirDem06, BirDem07, BirDem09, BirDem10}, which address fundamental results concerning the existence and uniqueness of solutions, maximum principles, and Harnack type inequalities.

When the right-hand side of \eqref{equation} is in $L^{\infty}$ and $F$ is a concave/convex operator, the best regularity one can expect is $\mathrm{C}^{1,\frac{1}{1+p}}_{\operatorname{loc}}$ as $p\geq 0$, see \cite{DRT01,IS}. However, through an iterative geometric approach, our first result, see Theorem \ref{improvement_regularity}, establishes a best pointwise regularity $\mathrm{C}^{1,\beta}$ along the set of free boundary points $\partial\{u>0\}$, whenever $0<\mu<p+1$ and $0<q<p+1$, since
$$
    \beta=\min\left\{\frac{p+2-q}{p+1-q},\frac{p+2}{p+1-\mu}\right\}>1+\frac{1}{p+1}.
$$
In other words, even if we add a Hamiltonian term, we obtain a pointwise regularity at free boundary that is better than the local optimum estimate for degenerate uniformly elliptic equations with the convex operator, which is the content of Theorem \ref{improvement_regularity}.

It is important to note that in the paper \cite{Leitao}, the authors approach a similar analysis, the difference being that the drift term they deal with $(\mathfrak{b}(x)\cdot\nabla u|\nabla u|^p)$ in an equation whose prototype is
$$
    |\nabla u|^pF(D^2u,x)+\mathfrak{b}(x)\cdot\nabla u|\nabla u|^p=\mathfrak{\lambda}_0(x)u(x)^{\mu}_+,
$$
has an agreement with the rescaling of the entire equation. Our case can be seen as a continuation of this work, since when $q=p+1$ in \eqref{equation}, the Hamiltonian term has the same order as the drift term in \cite{Leitao}. In this sense, considering $0\leq q<p+1$ brings significant difficulties, since the rescalings of the equation do not behave very well, which does not guarantee, for example, that the growth rate on the free boundary is optimal. However, the novelty of this work is to provide conditions to guarantee optimal growth on the free boundary through non-degeneracy, see Theorem \ref{ND_estimates}, which is not always valid depending on the exponents $\mu$ and $q$. But, if 
$$
    \frac{p+2}{p+1-\mu}\leq\frac{p+2-q}{p+1-q},
$$
the optimal growth rate along the boundary of the 
non-coincidence set $\partial\{u>0\}$ is obtained. Moreover, we show that in the case $\frac{p+2-q}{p+1-q}<\frac{p+2}{p+1-\mu}$, non-degeneracy can fail, and therefore, the growth on the free boundary may not be optimal.

Using a barrier argument together with Theorem \ref{improvement_regularity}, we derive a strong maximum principle, stated in Corollary \ref{strong.maximum.principle}, for the ``critical equation'' obtained when $\mu\rightarrow p+1$, that is
$$
    |\nabla u|^pF(D^2u)+\mathfrak{a}(x)|\nabla u|^q=\lambda_0(x)u(x)^{p+1}_+\quad\text{in}\quad\Omega,
$$
under suitable additional assumptions on the coefficients and data of the equation. We also have, as a consequence of Theorems \ref{improvement_regularity} and \ref{ND_estimates}, the uniform density of the region of positivity
near the free boundary, and information about the Hasdorff measure of the free boundary $\partial\{u>0\}$.

\subsection{Main results and some consequences}

We now start discussing the statement of the new improved regularity estimates proven in this paper.  In our first result, we establish a sharp and improved regularity estimate for viscosity solutions to \eqref{equation} at free boundary points.

\begin{theorem}\label{improvement_regularity}
Let $u$ be a non-negative and bounded viscosity solution to \eqref{equation} and consider $x_0\in\partial\{u>0\}\cap\Omega'$, where $\Omega'\Subset\Omega$. Then, for $0<r<\min\left\{1,\frac{\operatorname{dist}(\Omega',\partial\Omega)}{2}\right\}$ and any $x\in B_r(x_0)\cap\{u>0\}$, there holds
\begin{equation}\label{growth}
u(x)\leq \mathrm{C}_0\cdot\max\left\{1,\|u\|_{L^{\infty}(\Omega)},\|\mathfrak{a}\|^{\frac{1}{1+p-q}}_{L^{\infty}(\Omega)},\|\lambda_0\|^{\frac{1}{p+1-\mu}}_{L^{\infty}(\Omega)}\right\}|x-x_0|^{\beta},
\end{equation}
where 
\begin{equation}\label{beta}
\beta=\min\left\{\frac{p+2-q}{p+1-q},\frac{p+2}{p+1-\mu}\right\},
\end{equation}
and $\mathrm{C}_0$ is a constant depending only on $n,\lambda,\Lambda,p,q,\mu,\|\mathfrak{a}\|_{L^{\infty}(\Omega)},\|\lambda_0\|_{L^{\infty}(\Omega)}$, and $\operatorname{dist}(\Omega',\partial\Omega)$.
\end{theorem}

A geometric interpretation of Theorem \eqref{improvement_regularity} states that if $u$ solves \eqref{equation} and $x_0\in \partial\{u>0\}$, then for $x\in B_r(x_0)\cap \{u>0\}$, we have
$$
\sup_{B_r(x_0)}u(x)\leq \mathrm{C}r^\beta.
$$

From a geometric perspective, establishing a sharp lower bound estimate for the solutions is crucial, as it provides important qualitative information. This property is referred to as Non-degeneracy.

\begin{theorem}[Non-degeneracy estimate]\label{ND_estimates}
Let $u$ be a non-negative and bounded viscosity solution to \eqref{equation} in $B_1$. Then, there exists a universal constant $\mathrm{C}_1>0$ such that
\begin{equation}\label{Non-degeneracy}
\sup_{B_r(x_0)} u(x)\geq \mathrm{C}_1r^{\frac{p+2}{p+1-\mu}},
\end{equation}
whenever $\frac{p+2}{p+1-\mu}\leq\frac{p+2-q}{p+1-q}$, for any $x_0\in \overline{\{u>0\}}\cap B_{1/2}$ and $0<r<1/2$.
\end{theorem}

\begin{remark}
In the setting of this manuscript, non-degeneracy does not generally hold when \( \beta \) is given by Theorem \ref{improvement_regularity}. In Section \ref{ND_Section}, we provide an example that is a solution to problem \eqref{equation}, but non-degeneracy fails when \( \beta = \frac{p+2-q}{p+1-q}\).
\end{remark}

By standard argument, see \cite{DJ,Leitao,ET}, the combination of Theorems \ref{improvement_regularity} and \ref{ND_estimates} gives us a positive density result for the non-coincidence set, which, in turn, makes it impossible to form cusps on the free boundary pointing inward the coincidence set. In addition, an estimate of the Hausdorff dimension of the boundary of the coincidence set, ${\partial\{u>0\}}$, is also obtained. 

\begin{corollary}\label{Positivity}
    Let $u$ be a bounded, non-negative viscosity solution to \eqref{equation} in $\Omega$. Then, for  $x_0\in\partial\{u>0\}\cap\Omega'$ and any $0<r\ll 1$,
    $$
    \mathcal{L}^n(B_r(x_0)\cap\{u>0\})\geq \sigma r^n,
    $$
    for a universal constant $\sigma$ depending on $n,\lambda,\Lambda,p,q,\mu,\|\mathfrak{a}\|_{L^{\infty}(\Omega)},\|\lambda_0\|_{L^{\infty}(\Omega)}$, and $\operatorname{dist}(\Omega',\partial\Omega)$. Moreover, there exists a universal constant $\kappa>0$, such that
    $$
        \mathcal{H}^{n-\kappa}(\partial\{u>0\}\cap\Omega')<\infty.
    $$
    In other words, the Hasdorff dimension is at most $n-\kappa$.
\end{corollary}
\begin{remark}
    In order to use the Theorem \ref{ND_estimates}, we are suppose in the Corollary \ref{Positivity} that $\frac{p+2}{p+1-\mu}\leq\frac{p+2-q}{p+1-q}$.
\end{remark}
\vspace{0.2cm}
The paper is organized as follows, in Section \ref{HD}  we define the notion of viscosity solution and some more definitions related to the operator $F$. In Section \ref{EUV} we provide the Comparison Principle  that makes it possible to prove existence, as well as to obtain Non-degeneracy and positivity of solutions through a barrier argument. In Section \ref{LLE}, we obtain the Lipschitz estimate, which provides compactness to prove the Flatness Lemma in Section \ref{IRE}, and therefore growth on the free boundary (Theorem \ref{improvement_regularity}) is also  obtained in Section \ref{IRE}. Also in Section \ref{IRE}, we provide the positivity result (Corollary \ref{strong.maximum.principle}). Finally, Section \ref{ND_Section} is devoted to the proof of Theorem \ref{ND_estimates}. 
\section{Hypothesis and definitions}\label{HD}

In this section, we introduce the main structural hypothesis of the problem, along with the relevant definitions regarding function spaces and the notion of solution that will be used throughout this work.

Throughout this work we will suppose the following structural hypothesis:
\begin{assumption}
[\textbf{(Uniform ellipticity)}.] \label{unif ellipticity} There exist constants $\Lambda\geq\lambda>0$ such that 
\begin{equation}
\mathcal{P}^-_{\lambda,\Lambda}(N) \leq F(M+N) - F(M) \leq \mathcal{P}^+_{\lambda,\Lambda}(N).
\end{equation}
for all $M, N \in \text{Sym}(n)$ with $N \geq 0$. Here, $\operatorname{Sym}(n)$ denotes the space of symmetric matrices and $\mathcal{P}^{\pm}$ denotes the extremal Pucci operators
\begin{align*}
\mathcal{P}^+_{\lambda,\Lambda}(M) \coloneqq \Lambda\sum_{e_i>0} e_i + \lambda \sum_{e_i<0} e_i \quad\text{and}\quad
\mathcal{P}^-_{\lambda,\Lambda}(M) \coloneqq \lambda\sum_{e_i>0} e_i + \Lambda \sum_{e_i<0} e_i,
\end{align*}
where $(e_i)_{i=1}^n$ denotes are the eigenvalues of the matrix $M \in \mbox{Sym}(n)$.  We can express this as follows:
$$
\mathcal{P}^+_{\lambda,\Lambda}(M)=\sup_{A\in\mathcal{S}_{\lambda,\Lambda}} \T(AM)\quad\text{and}\quad \mathcal{P}^-_{\lambda,\Lambda}(M)=\inf_{A\in\mathcal{S}_{\lambda,\Lambda}} \T(AM),
$$
for $\mathcal{S}_{\lambda,\Lambda}\coloneqq\{A\in \mbox{Sym}(n): \lambda I\leq A\leq \Lambda I\}$.
\end{assumption}

We now introduce the notion of viscosity solution associated with our operators.

\begin{definition}
We say that $u \in \mathrm{C}^0(\Omega)$ is a viscosity super-solution (resp. sub-solution) to
$$
\mathcal{G}(x,\nabla u, D^2u)=f(x,u)\quad\text{in}\quad \Omega
$$
if for all $x_0\in \Omega$ we have the following 
\begin{enumerate}
\item[$1.$] either for all $\phi\in \mathrm{C}^2(\Omega)$ such that $u-\phi$ attains a local minimum at $x_0$ and $|\nabla\varphi(x_0)|\neq 0$, one has
$$
\mathcal{G}(x_0,\nabla \phi(x_0),D^2\phi(x_0))\leq f(x_0,\phi(x_0))\quad (\text{resp.} \,\, \geq f(x_0,\phi(x_0)))
$$
\item[$2.$] or there exists an open ball $B_\iota(x_0)\subset \Omega$, for $\iota>0$ where $u$ is constant, $u=K$ and there holds
$$
f(x_0,K)\geq 0 \quad (\text{resp.} \,\, f(x_0,K)\leq 0).
$$
\end{enumerate}
Finally, we say that $u$ is a viscosity solution if it is viscosity super- and sub-solution simultaneously.
\end{definition}

We now define the super/sub-jets of a function $u$ at
the point $x$, as introduced by Crandall, Ishii, and Lions in \cite{CIL}.

\begin{definition}
A second-order superjet of \( u \) at \( x_0 \in \Omega \) is defined as
\[
\mathcal{J}^{2,+}_{\Omega} u(x_0) = \left\{ (\nabla \phi(x_0), D^2 \phi(x_0)) : \phi \in \mathrm{C}^2 \text{ and } u - \phi \text{ attains a local maximum at } x_0 \right\}.
\]
The closure of a superjet is given by
\begin{eqnarray*}
\overline{\mathcal{J}}^{2,+}_{\Omega} u(x_0) = \left\{ (\vec{\xi}, \mathrm{X}) \in \mathbb{R}^n \times \text{Sym}(n): \exists\,\, (\vec{\xi}_k, \mathrm{X}_k) \in \mathcal{J}^{2,+}_{\Omega} u(x_k) \text{ such that}\right.\\ 
\left.(x_k, u(x_k), \vec{\xi}_k, \mathrm{X}_k) \to (x_0, u(x_0), \vec{\xi}, \mathrm{X}) \right\}.
\end{eqnarray*}
The second-order subjet of $u$ at $x_0$ is defined by $\mathcal{J}^{2,-}_{\Omega}u(x_0)\coloneqq-\mathcal{J}^{2,+}_{\Omega}(-u)(x_0)$. Then, we can define $\overline{\mathcal{J}}^{2,-}_{\Omega}u(x_0)$.
\end{definition}

\section{Existence/Uniqueness of viscosity solutions}\label{EUV}

Given a non-negative boundary data \( g \in \mathrm{C}^0(\partial\Omega) \) and a continuous function \( h \in \mathrm{C}^0(\overline{\Omega}) \), we comment on the existence and uniqueness of solutions to the problem \eqref{equation}. The comparison principle we established below for the Dirichlet problem
\begin{equation}\label{Dirichlet_Problem}
\left\{
\begin{array}{rllll}
|\nabla u|^pF(D^2u)+\mathfrak{a}(x)|\nabla u|^q-\lambda_0(x)u^{\mu}_+ &= & h(x) &\mbox{in} & \Omega,\\[0.2cm]
u &=& g &\mbox{on} & \partial\Omega,
\end{array}
\right.
\end{equation}
plays a fundamental role in establishing the existence of our problem, but also in establishing geometric properties that will be discussed later. Our approach is inspired by \cite{DJ, BV} with an extra step to deal with the Hamiltonian
term.

\begin{lemma}[Comparison Principle]\label{ComparisonPrin}
Let $u,v\in \mathrm{C}^0(\overline{\Omega})$ and let $h_1,h_2\in C^0(\overline{\Omega})$ satisfying
\begin{equation}\label{CP}
\begin{array}{rllll}
|\nabla u|^pF(D^2u)+\mathfrak{a}(x)|\nabla u|^q-\lambda_0(x)u^{\mu}_+ &\geq & h_1(x) &\mbox{in} & \Omega,\\[0.2cm]
|\nabla v|^pF(D^2v)+\mathfrak{a}(x)|\nabla v|^q-\lambda_0(x)v^{\mu}_+ &\leq & h_2(x) &\mbox{in} & \Omega.
\end{array}
\end{equation}
in the viscosity sense. Furthemore, suppose that either $u_1$ or $u_2$ is locally Lipschitz continuous in $\Omega$, and also that one of the following holds
\begin{itemize}
    \item[(i)] $h_1>h_2$ in $\overline{\Omega}$ and $\inf_{\Omega}\lambda_0\geq 0$;
    \item[(ii)] $h_1\geq h_2$ in $\overline{\Omega}$ and $\inf_{\Omega}\lambda_0>0$.
\end{itemize}
If $v\geq u$ on $\partial\Omega$, then $v\geq u$ in $\Omega$. 
\end{lemma}

\begin{proof}
Let us suppose, for the purpose of contradiction, that exists $M_0>0$ such that
$$
M_0\coloneqq\sup_{\overline{\Omega}}(u-v).
$$
Now, for each $\varepsilon>0$, define
\begin{equation}\label{maximum}
M_\varepsilon = \sup_{\Bar{\Omega}\times \Bar{\Omega}} \left(u(x)-v(y)-\frac{1}{2\varepsilon}|x-y|^2\right)<\infty.
\end{equation}
Let $(x_\varepsilon,y_\varepsilon) \in \Bar{\Omega}\times \Bar{\Omega}$ be where the supremum $M_\varepsilon$ is attained. It follows as in \cite[Lemma 3.1]{CIL}, there holds
\begin{equation}\label{CP1}
\lim_{\varepsilon\rightarrow 0}\frac{1}{\varepsilon}|x_\varepsilon-y_\varepsilon|^2=0\quad\mbox{and}\quad \lim_{\varepsilon\rightarrow 0} M_\varepsilon = M_0. 
\end{equation}
In particular,
\begin{equation}\label{CP2}
z_0\coloneqq\lim_{\varepsilon\rightarrow 0} x_\varepsilon=\lim_{\varepsilon\rightarrow 0} y_\varepsilon,
\end{equation}
with $u(z_0)-v(z_0)=M_0$. Since,
$$
M_0>0\geq \sup_{\partial\Omega}(u-v),
$$
it implies $x_\varepsilon,y_\varepsilon\in \Omega'$ for some interior compact subdomain $\Omega'\Subset\Omega$ and $0<\varepsilon\ll1$. Therefore, by the Ishii-Lions Lemma, see \cite[Theorem 3.2]{CIL}, there exist $\mathcal{M},\mathcal{N}\in \mbox{Sym}(n)$ such that
\begin{equation}\label{CP3}
\left(\frac{x_\varepsilon-y_\varepsilon}{\varepsilon},\mathcal{M}\right)\in \overline{J}^{2,+}_{\Omega'} u(x_\varepsilon)\quad\mbox{and}\quad \left(\frac{y_\varepsilon-x_\varepsilon}{\varepsilon},\mathcal{N}\right)\in \overline{J}^{2,-}_{\Omega'} v(y_\varepsilon)
\end{equation}
and
\begin{equation}\label{CP4}
-\frac{3}{\varepsilon} \left(
\begin{array}{cc}
I & 0 \\
0 & I
\end{array}
\right)\leq \left(
\begin{array}{cc}
\mathcal{M} & 0 \\
0 & -\mathcal{N}
\end{array}
\right)\leq \frac{3}{\varepsilon} \left(
\begin{array}{cc}
I & -I \\
-I & I
\end{array}
\right).
\end{equation}
In particular, we obtain that $\mathcal{M}\leq\mathcal{N}$. 

Now, we claim that there exists a constant $\mathcal{L}$ such that
\begin{equation}\label{ComPrinClaim}
    \frac{|x_{\varepsilon}-y_{\varepsilon}|}{\varepsilon}\leq C,
\end{equation}
for any $\varepsilon\ll 1$.  Indeed, since that $u$ or $v$ is locally Lipschitz in $\Omega'$, there exist a positive constant $\mathcal{K}$ such that
\begin{equation}\label{sup inequality}
|u(x)-v(y)|\leq\sup_{\Omega'}(u_2-u_1)+\mathcal{K}|x-y|,\quad\text{for all}\,\, x,y\in\Omega'.
\end{equation}
Define
$$
\psi_{\varepsilon}(x,y)\coloneqq u(x)-v(y)-\frac{1}{2\varepsilon}|x-y|^2,
$$
and note that
$$
\sup_{\Omega'}(u-v)\leq\sup_{\Omega'\times\Omega'}\psi_{\varepsilon}(x,y)= \psi_{\varepsilon}(x_{\varepsilon},y_{\varepsilon}).
$$
Then, using \eqref{sup inequality}, we obtain
\begin{align*}
\sup_{\Omega'}(u-v)+\frac{1}{2\varepsilon}|x_{\varepsilon}-y_{\varepsilon}|^{2}&\leq \psi_{\varepsilon}(x_{\varepsilon},y_{\varepsilon})+\frac{1}{2\varepsilon}|x_{\varepsilon}-y_{\varepsilon}|^2\\
&\leq \sup_{\Omega'}(u-v)+\mathcal{K}|x_{\varepsilon}-y_{\varepsilon}|.
\end{align*}
This concludes the proof of inequality \eqref{ComPrinClaim}.

Finally, we will conclude the proof of the lemma. By \eqref{CP} and \eqref{CP3}, we obtain
$$
h_1(x_\varepsilon)+\lambda_0(x_\varepsilon)u(x_\varepsilon)_+^\mu \leq \left|\frac{x_\varepsilon-y_\varepsilon}{\varepsilon}\right|^pF(\mathcal{M})+\mathfrak{a}(x_\varepsilon)\left|\frac{x_\varepsilon-y_\varepsilon}{\varepsilon}\right|^q 
$$
and
$$
h_2(y_\varepsilon)+\lambda_0(y_\varepsilon)v(y_\varepsilon)_+^\mu  \geq \left|\frac{x_\varepsilon-y_\varepsilon}{\varepsilon}\right|^pF(\mathcal{N})+\mathfrak{a}(y_\varepsilon)\left|\frac{x_\varepsilon-y_\varepsilon}{\varepsilon}\right|^q.
$$
Consequently, using \eqref{CP4} and \eqref{ComPrinClaim}, there holds
\begin{align*}
h_1(x_\varepsilon)-h_2(y_\varepsilon) & \leq\left|\frac{x_\varepsilon-y_\varepsilon}{\varepsilon}\right|^p\left(F(\mathcal{M})-F(\mathcal{N})\right)+ \left|\frac{x_\varepsilon-y_\varepsilon}{\varepsilon}\right|^q(\mathfrak{a}(x_\varepsilon)-\mathfrak{a}(y_\varepsilon))\\[0.2cm]
& \,\,\,\,\,\,\,\,\,\,-\lambda_0(x_\varepsilon)u(x_\varepsilon)_+^\mu+\lambda_0(y_\varepsilon)v(y_\varepsilon)_+^\mu \\[0.2cm]
&\leq  \mathcal{K}^q(\mathfrak{a}(x_\varepsilon)-\mathfrak{a}(y_\varepsilon))-\lambda_0(x_\varepsilon)u(x_\varepsilon)_+^\mu+\lambda_0(y_\varepsilon)v(y_\varepsilon)_+^\mu.
\end{align*}
Therefore, using \eqref{maximum}, we have
\begin{align*}
    h_1(x_\varepsilon)-h_2(y_\varepsilon)\leq  \mathcal{K}^q&(\mathfrak{a}(x_\varepsilon)-\mathfrak{a}(y_\varepsilon))+\lambda_0(y_\varepsilon)v(y_\varepsilon)_+^\mu\\
    & -\left(M_\varepsilon+v(y_\varepsilon)+\frac{1}{2\varepsilon}|x_\varepsilon-y_\varepsilon|^2\right)_+^\mu\lambda_0(x_\varepsilon)
\end{align*}
Finally, by continuity of $\lambda_0$ and $\mathfrak{a}$ together with \eqref{CP1} and \eqref{CP2}, and letting $\varepsilon\rightarrow 0$ in the estimate above, it gives
\begin{align*}
    h_1(z_0)-h_2(z_0)&\leq\lambda_0(z_0)(v(z_0)^{\mu}_+-[M_0+v(z_0)]_+^\mu)\\
    &\leq\inf_{\Omega}\lambda_0(v(z_0)^{\mu}_+-[M_0+v(z_0)]_+^\mu).
\end{align*}
Which contradicts assumptions (i) and (ii), since $v\geq 0$ and $M_0>0$.
\end{proof}

We briefly comment on the existence of a viscosity solution to the Dirichlet problem \eqref{equation}. As is customary, it follows from an application of Perron's method, provided that the comparison principle holds. Let us consider $\overline{u}$ and $\underline{u}$ solutions to the following boundary value problems
$$
\left\{
\begin{array}{rccccc}
  |\nabla \overline{u}|^pF(D^2\overline{u})+\mathfrak{a}(x)|\nabla \overline{u}|^q   & = & 0 &\text{in} & \Omega, \\
    \overline{u} & =& g & \text{on}&\partial\Omega,
\end{array}
\right.
$$
and
$$
\left\{
\begin{array}{rclccc}
  |\nabla \underline{u}|^pF(D^2\underline{u})+\mathfrak{a}(x)|\nabla \underline{u}|^q   & = & \|g\|^{\mu}_{L^\infty(\partial\Omega)} &\text{in} & \Omega, \\
    \underline{u} & =& g & \text{on}&\partial\Omega.
\end{array}
\right.
$$
The existence of such solutions follows standard arguments. Moreover, note that $\overline{u}$ and $\underline{u}$ are respectively, supersolution and subsolution to \eqref{equation}. Therefore we can apply the Comparison principle, Lemma \ref{ComparisonPrin}  above, it is possible, under a direct application of Perron’s method, to obtain the existence of a viscosity solution in $\mathrm{C}^0(\Omega)$ of \eqref{Dirichlet_Problem}. More precisely we have the following theorem
\begin{theorem}[Existence and uniquiness]
Let \( g \in \mathrm{C}^0(\partial\Omega) \) be a non-negative boundary data and \( h \in \mathrm{C}^0(\overline{\Omega}) \). Suppose that there exist $\underline{u}, \overline{u}\in \mathrm{C}^0(\overline{\Omega})\cap \mathrm{C}^{0,1}(\Omega)$ are, respectively a viscosity subsolution and supersolution of \eqref{equation} satisfying $\underline{u}, \overline{u}=g$ on $\partial\Omega$. Define the class of functions
$$
\mathcal{S}(\Omega)\coloneqq \{\omega(x): \,\, v\,\, \text{is a supersolution to } \eqref{equation}\,\, \text{and} \,\, \underline{u}\leq \omega\leq \overline{u}\}.
$$
Then,
$$
u(x)=\inf_{\mathcal{S}(\Omega)}\omega(x),
$$
is the unique continuous viscosity solution to \eqref{equation}.
\end{theorem}

\section{Local Lipschitz estimates}\label{LLE}
    
In this section, we establish local Lipschitz estimate for viscosity solution of \eqref{equation}. Before presenting the result, observe that, without loss of generality, \(u\), \(\mathfrak{a}\), and \(\lambda_{0}\) can be considered normalized; otherwise, it suffices to consider the new function
$$
v(x)=\frac{u(x)}{\kappa},
$$
for 
$$
\kappa\coloneqq\max\left\{1,\|u\|_{\infty},\|\mathfrak{a}\|^{\frac{1}{p+1-q}}_{\infty},\|\lambda_0\|^{\frac{1}{p+1-\mu}}_{\infty}\right\}.
$$
Moreover, since the estimate is local, we shall henceforth take our domain to be the unit ball \(B_{1}\).

\begin{lemma}[Lipschitz continuity]\label{lipschitzcontinuity}
Let $u\in \mathrm{C}^0(B_1)$ be a viscosity solution of \eqref{equation}. Then, $u\in \mathrm{C}^{0,1}_{loc}(B_1)$ and 
$$
\|u\|_{\mathrm{C}^{0,1}(B_{1/2})}\leq \mathrm{C}\cdot\max\left\{1,\|u\|_{\infty},\|\mathfrak{a}\|^{\frac{1}{p+1-q}}_{\infty},\|\lambda_0\|^{\frac{1}{p+1-\mu}}_{\infty}\right\},
$$
where $\mathrm{C}$ is a universal constant depending only on $n,\lambda,\Lambda,p$, and $q$. 
\end{lemma}

\begin{proof}
Here, we use techniques introduced in \cite{IL}. Let us fix $0<r<1/2$ and for constants $\overline{L},\varrho$, we define $\varphi:\overline{B}_r(x_0)\times\overline{B}_r(x_0)\rightarrow\mathbb{R}$ given by
$$
\varphi(x,y)\coloneqq \overline{L}v(|x-y|)+\varrho(|x-x_0|^2+|y-x_0|^2)
$$
for each $x_0\in B_{1/2}$, where
    \begin{equation*}
    v(s)=
    \begin{cases}
        s-\omega_0s^{\frac{3}{2}}\quad \text{if }\,\, 0\leq s\leq s_0,\\
        v(s_0)\quad\quad\,\,\,\, \text{if }\,\,s>s_0.
    \end{cases}
    \end{equation*}
We choose $\omega_0$ such that $s_0:=(2/3\omega_0)^2\geq 1$, then we can fix $\omega_0=1/3$. Now consider
$$
    M(x_0)=\sup_{x,y\in \overline{B}_{r}(x_0)}\left\{u(x)-u(y)-\varphi(x,y)\right\}.
$$
To prove the local Lipschitz regularity of $u$, we show that there exist two constants $\overline{L}$, $\varrho$ such that $M(x_0)\leq 0$. We argue by contradiction that $M(x'_0)>0$ for some $x'_0\in B_{1/2}$ and for any $\overline{L},\varrho>0$. Let $(\overline{x}, \overline{y})\in\overline{B}_r(x'_0)\times\overline{B}_r(x'_0)$ be the pair where $M$ is attained. We observe that since $M>0$
\begin{equation}\label{inter_point}
\overline{L}v(|\overline{x}-\overline{y}|)+\varrho(|\overline{x}-x'_0|^2+|\overline{y}-x'_0|^2)<u(\overline{x})-u(\overline{y})\leq 2.
\end{equation}
If we choose $\varrho=\left(\frac{4\sqrt{2}}{r}\right)^2$ and $\overline{L}\geq \varrho$, we obtain that $\overline{x}$ and $\overline{y}$ are interior maximum point, and $\overline{x}\neq \overline{y}$, otherwise, $M\leq 0$. We use the Ishii--Lions lemma \cite[Theorem 3.2]{CIL} to guarantee the existence of limiting sub-jet and super-jet
$$
(\xi_{\overline{x}},X)\in \overline{\mathcal{J}}^{2,+}_{B_{1/2}(x'_0)}u(\overline{x})\quad\text{and}\quad (\xi_{\overline{y}},Y)\in \overline{\mathcal{J}}^{2,-}_{B_{1/2}(x'_0)} u(\overline{y}),
$$ 
where
$$
\xi_{\overline{x}}\coloneqq D_x\varphi(\overline{x},\overline{y})=\overline{L}v'(|\overline{x}-\overline{y}|)\frac{\overline{x}-\overline{y}}{|\overline{x}-\overline{y}|}+2\varrho(\overline{x}-x'_0)
$$
and
$$
    \xi_{\overline{y}}\coloneqq-D_y\varphi(\overline{x},\overline{y})=\overline{L}v'(|\overline{x}-\overline{y}|)\frac{\overline{x}-\overline{y}}{|\overline{x}-\overline{y}|}-2\varrho(\overline{y}-x'_0)
$$
such that the matrices $X$ and $Y$ verify the inequality
\begin{equation}\label{matrix_inequality1}
\left(
\begin{array}{ccc}
X   & 0 \\
0  &-Y 
\end{array}
\right)
\leq 
\left(
\begin{array}{ccc}
Z  & -Z \\
-Z  & Z
\end{array}
\right)+(2\varrho+\kappa)I,
\end{equation}
for $\kappa>0$ depends only on the norm of $Z$, can be sufficiently small, and
$$
Z\coloneqq\overline{L}\left[\frac{v'(|\overline{x}-\overline{y}|)}{|\overline{x}-\overline{y}|}I+\left(v''(|\overline{x}-\overline{y}|)-\frac{v'(|\overline{x}-\overline{y}|)}{|\overline{x}-\overline{y}|}\right)\frac{(\overline{x}-\overline{y})\otimes(\overline{x}-\overline{y})}{|\overline{x}-\overline{y}|^2}\right].
$$
By the matrix inequality \eqref{matrix_inequality1}, we can deduce
$$
    \langle(X-Y)z,z\rangle\leq(4\varrho+2\kappa)|z|^2,
$$
for any vectors of the form $(z,z)\in\mathbb{R}^{2n}$, that is, all the eigenvalues of $(X-Y)$ are below $4\varrho+2\kappa$. On the other hand, applying \eqref{matrix_inequality1} again, now to the vector
$$
\overline{z}\coloneqq\left(\frac{\overline{x}-\overline{y}}{|\overline{x}-\overline{y}|},\frac{\overline{y}-\overline{x}}{|\overline{x}-\overline{y}|}\right),
$$
we get
$$
    \left\langle(X-Y)\overline{z},\overline{z}\right\rangle\leq(4\varrho+2\kappa+4\overline{L}v''(|\overline{x}-\overline{y}|))|\overline{z}|^2.
$$
Then, at least one eigenvalue of $(X-Y)$ is below $4\varrho+2\kappa+4\overline{L}v''(|\overline{x}-\overline{y}|)$, this quantity will be negative when $\overline{L}\geq \frac{4\varrho+2}{3\omega_0}$. Therefore, we obtain
\begin{align}\label{P+}
    \mathcal{P}^+(X-Y) &\leq  \Lambda(n-1)(4\varrho+2\kappa)+\lambda(4\varrho+2\kappa+4\overline{L}v''(\overline{x}-\overline{y}))\nonumber\\
    &= 2(\lambda+\Lambda(n-1))(2\varrho+\kappa)+4\lambda \overline{L}v''(|\overline{x}-\overline{y}|).
\end{align}
In addition, we have the following inequalities in the viscosity sense 
\begin{equation}
\label{inequalityVS}
\left\{
\begin{array}{cccc}
|\xi_{\overline{x}}|^pF(X)+\mathfrak{a}(\overline{x})| \xi_{\overline{x}}|^q & \geq & \lambda_0(\overline{x}) u(\overline{x})^{\mu},\\
|\xi_{\overline{y}}|^pF(Y)+\mathfrak{a}(\overline{y})| \xi_{\overline{y}}|^q & \leq &\lambda_0(\overline{y}) u(\overline{y})^{\mu}.
\end{array}
\right.
\end{equation}
Note that, $|\xi_{\overline{x}}|>1$ and similarly $|\xi_{\overline{y}}|>1$, because
\begin{equation}\label{des1}
|\xi_{\overline{x}}| \geq \overline{L}v'(|\overline{x}-\overline{y}|)-2\varrho|\overline{x}-x'_0|
\geq \frac{1}{2}\left(\frac{4\varrho+2}{3\omega_0}-\varrho\right)\\
    = \varrho+1\\
    > 1.
\end{equation}
In addition, the following inequality holds
\begin{equation}
    |\xi_{\overline{x}}|,|\xi_{\overline{y}}|\leq \overline{L}+2\varrho\leq 2\overline{L}\label{des2}.
\end{equation}
From \eqref{inequalityVS}, we have
\begin{equation}
\label{inequalityVSReit}
\begin{cases}
F(X)\geq \lambda_0(\overline{x}) u(\overline{x})^{\mu}|\xi_{\overline{x}}|^{-p} -\mathfrak{a}(\overline{x})| \xi_{\overline{x}}|^{q-p},\\
F(Y)\leq \lambda_0(\overline{y}) u(\overline{y})^{\mu}|\xi_{\overline{y}}|^{-p}-\mathfrak{a}(\overline{y})| \xi_{\overline{y}}|^{q-p}.
\end{cases}
\end{equation}
Hence, using the inequalities \eqref{P+} and \eqref{inequalityVSReit}, 
we obtain
\begin{align*}
\lambda_0(\overline{x}) u(\overline{x})^{\mu}|\xi_{\overline{x}}|^{-p}-\mathfrak{a}(\overline{x})|\xi_{\overline{x}}|^{q-p} &\leq F(X)\\
&\leq F(Y)+\mathcal{P}^+(X-Y)\\
& \leq \{\lambda_0(\overline{y}) u(\overline{y})^{\mu}|\xi_{\overline{y}}|^{-p}-\mathfrak{a}(\overline{y})|\xi_{\overline{y}}|^{q-p}\}\\
& +2(\lambda+\Lambda(n-1))(2\varrho+\kappa)-\lambda \overline{L},
\end{align*}
consequently
$$
\lambda \overline{L}\leq 2A_1+2A_2+2(\lambda+\Lambda(n-1))(2\varrho+\kappa)
$$
where $A_1=\max\{1,2^{-p}\overline{L}^{-p}\}$ and $A_2=\max\{1,2^{q-p}\overline{L}^{q-p}\}$, the definition of $A_1$ and $A_2$ follows from the inequalities \eqref{des1} and \eqref{des2}. Therefore, we get a contradiction when $\overline{L}$ is large enough since $-p<1$ and $q-p<1$.
\end{proof}

\section{Improved regularity estimate: Proof of Theorem \ref{improvement_regularity}}\label{IRE}

In this section, we obtain a sharp and improved regularity for non-negative solutions to \eqref{equation} along the free boundary \( \partial\{u > 0\} \). To this end, a geometric iteration is developed as in \cite{DJ,Leitao,ET} with suitable adaptations to our setting, which is possible through the flatness improvement below.

\begin{lemma}\label{estimaHA}
Given $\varepsilon\in (0,1)$, there exists $\delta(\varepsilon)>0$ depending on $n,\lambda,\Lambda,p,q,\mu$, such that if $u$ solves \eqref{equation} in $B_1$, with $0\leq u\leq 1$, $u(0)=0$, and
$$
\|\mathfrak{a}\|_{L^{\infty}(B_1)}+\|\lambda_0\|_{L^{\infty}(B_1)}\leq\gamma,
$$
with $0<\gamma\leq\delta(\varepsilon)$, then
$$
\sup_{B_{1/2}}u(x)\leq\varepsilon.
$$
\end{lemma}
\begin{proof}
We proceed by contradiction. Suppose that the statement does not hold. Then, there exist $\varepsilon_0>0$ and sequences of functions $\{F_k\}_{k\in\mathbb{N}}$, $\{u_k\}_{k\in\mathbb{N}}$, $\{\mathfrak{a}_k\}_{k\in\mathbb{N}}$, $\{\lambda_k\}_{k\in\mathbb{N}}$ satisfying
\begin{eqnarray}
0\leq u_k\leq 1, \quad \text{with}\quad u_k(0)=0,\label{IF1}\\
\|\mathfrak{a}_k\|_{L^{\infty}(B_1)}+\|\lambda_k\|_{L^{\infty}(B_1)}\leq k^{-1},\label{IF2}
\end{eqnarray}
    and
\begin{equation}\label{IF3}
|\nabla u_k|^pF_k(D^2u_k)+\mathfrak{a}_k(x)|\nabla u_k|^q=\lambda_k(x)(u_k)^{\mu}_+,
\end{equation}
for $\lambda_k(x)\geq 0$ and $F_k$ are $(\lambda,\Lambda)$-elliptic operators. However
\begin{equation}\label{stabilityContradiction} 
\sup_{B_{1/2}}u_k>\varepsilon_0.
\end{equation}
Observe that, by uniform ellipticity and, if necessary, passing to a subsequence, we have
$$
F_k(M)\rightarrow F_{\infty}(M) \quad\text{locally uniformly in}\,\, \M(n).
$$ 
From Lipschitz local estimate, Lemma \ref{lipschitzcontinuity}, there exist $u_{\infty}$ such that, up to subsequence,
$$
u_k\to u_{\infty}\quad \text{locally uniformly in}\,\, B_1.
$$
Moreover, it follows that $u_{\infty}(0)=0$ and $0\leq u_{\infty}\leq 1$. 

Now we proceed to prove that $u_{\infty}$ solves in the viscosity sense
\begin{equation}\label{stabilityLemma}
    F_{\infty}(D^2 u_{\infty})=0\quad\text{in } B_{3/4}.
\end{equation}
For that, we only deduce that $u_{\infty}$ is a supersolution, since property of subsolution is derived in a similar way. In this sense, consider a test function $\varphi$ that touches $u_{\infty}$ from below at $x\in B_{3/4}$. For simplicity, we can assume that $|x|=u_{\infty}(0)=0$ and $\varphi$ is a quadratic polynomial, then
$$
\varphi(x)=\frac{1}{2}\langle Mx,x\rangle+\langle \mathcal{B},x\rangle.
$$
By the uniform convergence of $u_j$ to $u_{\infty}$, there exist $x_j\in B_r$ for small $r>0$, such that
$$
\varphi_j(x)\coloneqq \frac{1}{2}\langle M(x-x_j),(x-x_j)\rangle+\langle \mathcal{B},x-x_j\rangle +u_j(x_j)
$$
touches $u_j$ from below at point $x_j$. Since $u_j$ is a viscosity solution to \eqref{IF3}, if $\mathcal{B}\neq 0$, we obtain
$$
|\mathcal{B}|^p F_j(M)+\mathfrak{a}_j(x_j)|\mathcal{B}|^q\leq \lambda_j(x_j)u^{\mu}_j(x_j).
$$
Then, letting $j\rightarrow\infty$, it implies that $F_{\infty}(M)\leq 0$.

On the other hand, when $\mathcal{B}=0$, we want to justify that $F_{\infty}(M)\leq 0$. Assume by contradiction that $F_{\infty}(M)>0$, since $F_{\infty}$ is uniformly elliptic, it follows that $M$ has at least one positive eigenvalue. Let $\mathbb{R}^n$=$E\oplus L$ be the orthogonal sum, where $E=\operatorname{span}(e_1,...,e_k)$ is the subspace generated by the eigenvectors corresponding to the positive eigenvalues. Now, let $P_E$ be the orthogonal projection on $E$, and define the test function below 
$$
    \psi(x)\coloneqq \varphi(x)+\gamma|P_E(x)|=\frac{1}{2}\langle Mx,x\rangle+\gamma|P_E(x)|
$$
Since $u_j\rightarrow u_{\infty}$ locally uniformly in $B_1$ and $\varphi$ touches $u$ from below at $0$, then for $\gamma$ small enough, $u_j-\psi$ has a minimum at a point $x^{\gamma}_j\in B_r$. We also have that, up to subsequence, $x^{\gamma}_j\rightarrow x_0$ as $j\rightarrow\infty$. 

Assume that $P_E(x^{\gamma}_j)=0$ for a infinite quantity of points $x^{\gamma}_j$. We claim that $F_{\infty}(M)\leq 0$, generating a contradiction. To this end, note that
$$
    \bar{\psi}(x)\coloneqq \varphi(x)+\gamma e\cdot P_E(x)
$$
touches $u_j$ from below at $x^{\gamma}_j$ for each $e\in \mathbb{S}^{n-1}$. Moreover, suppose that $Mx_0=0$. Then, since $u_j$ is a viscosity solution, for $1\ll j$, we have
$$
|Mx^{\gamma}_j+\gamma e|^p F_j(M)+\mathfrak{a}_j(x^{\gamma}_j)|Mx^{\gamma}_j+\gamma e|^q\leq\lambda_j(x^{\gamma}_j)u^{\mu}_j(x^{\gamma}_j)
$$
for $e\in E\cap\mathbb{S}^{n-1}$. Thus, letting $j\rightarrow\infty$, it follows that $F_{\infty}(M)\leq 0$.
Now, if $Mx_0>0$, we take $\gamma<\frac{1}{4}|Mx_0|$, then $|Mx_0+\gamma e|>0$ as $j\rightarrow\infty$. This implies that $F_{\infty}(M)\leq 0$.

Next, we consider the case $P_E(x^{\gamma}_j)\neq 0$. Note that, $\psi$ is smooth in a neighbourhood of $x^{\gamma}_j$ and $x\mapsto |P_E(x)|$ is convex, which implies in the following viscosity inequality
\begin{align*}
\lambda_j(x^{\gamma}_j)u^{\mu}_j(x^{\gamma}_j)&\geq |Mx^{\gamma}_j+\gamma \nu^{\gamma}_j|^pF_j(M+\gamma B)+\mathfrak{a}_j(x^{\gamma}_j)|Mx^{\gamma}_j+\gamma \nu^{\gamma}_j|^q\\
& \geq |Mx^{\gamma}_j+\gamma \nu^{\gamma}_j|^pF_j(M)+\mathfrak{a}_j(x^{\gamma}_j)|Mx^{\gamma}_j+\gamma \nu^{\gamma}_j|^q
\end{align*}
where
$$\nu^{\gamma}_j\coloneqq\frac{P_E(x^{\gamma}_j)}{|P_E(x^{\gamma}_j)|}\quad\text{and}\quad B\coloneqq|P_E(x^{\gamma}_j)|^{-1}(I-\nu^{\gamma}_j\otimes\nu^{\gamma}_j).$$
Therefore, by writing $e=\nu^{\gamma}_j$, it is enough to proceed as in the case $P_E(x^{\gamma}_j)=0$ to prove by contradiction that $F_{\infty}(M) \leq 0$, via distinguishing $Mx_0 = 0$ and $Mx_0\neq0$.

Finally, we conclude that $u_{\infty}$ solves \eqref{stabilityLemma} in the viscosity sense. Thus, from Harnack inequality, see \cite[Theorem 4.3]{Caffa-Cabre}, we obtain $u_{\infty}\equiv0$ in $B_{2/3}$, contradicting \eqref{stabilityContradiction}.
\end{proof}

Finally, we are ready to prove the Theorem \ref{improvement_regularity}.

\begin{proof}[Proof of Theorem \ref{improvement_regularity}]
Set 
$$
v_1(x)=\frac{u(x_0+\rho x)}{\tau}\quad \text{with}\quad x\in B_1,
$$ 
for $\tau,\rho>0$ constants to be chosen later. As $u$ solves \eqref{equation}, we can verify that $v_1$ solves
\begin{equation}\label{Eq.v1}
|\nabla v|^pF_1(D^2v)+\mathfrak{a}_1(x)|\nabla v|^q=\lambda_1(x)v^{\mu}_+\quad\text{in}\,\, B_1,
\end{equation}
in the viscosity sense, where
\begin{align*}
F_1(M)  &\coloneqq  \left(\frac{\tau}{\rho^2}\right)^{-1}F\left(\frac{\tau}{\rho^2}M\right),\\[0.2cm]
\mathfrak{a}_1(x)  &\coloneqq  \frac{\rho^{p-q+2}}{\tau^{p-q+1}}\mathfrak{a}(x_0+\rho x),\\
\lambda_1(x)&\coloneqq \frac{\rho^{p+2}}{\tau^{p+1-\mu}}\lambda_0(x_0+\rho x).    
\end{align*}
Note that $F_1$ is still a $(\lambda,\Lambda)$-elliptic operator. We then make the following choices
\begin{align*}
    \tau \coloneqq \max\left\{1,\|u\|_{\infty},(2\|\mathfrak{a}\|_{\infty})^{\frac{1}{1+p-q}},(2\|\lambda_0\|_{\infty})^{\frac{1}{p+1-\mu}}\right\} \quad\text{and}\quad
    \rho \coloneqq \min\left\{\frac{\operatorname{dist}(\Omega',\partial\Omega)}{2},\gamma^{\frac{1}{p+2-q}}\right\},
\end{align*}
where $\gamma>0$. With that, we get
\begin{equation*}
\begin{array}{ccccc}
\|\mathfrak{a}_1(x)\|_{L^{\infty}(B_1)}\leq \frac{\gamma}{2\|\mathfrak{a}\|_{\infty}}\|\mathfrak{a}(x_0+\rho x)\|_{L^{\infty}(B_1)}\leq\frac{\gamma}{2};\\[0.3cm]
\|\lambda_1\|_{L^{\infty}(B_1)}\leq\frac{\gamma^{\frac{p+2}{2+p-q}}}{2\|\lambda_0\|_{\infty}}\|\lambda_0(x_0+\rho x)\|_{L^{\infty}(B_1)}\leq\frac{\gamma}{2};\\[0.3cm]
\|v_1(x)\|_{L^{\infty}(B_1)}\leq 1\quad\text{and}\quad v_1(0)=0.
\end{array}
\end{equation*}
Thus, we can apply the Lemma \ref{estimaHA} to $v_1$, which guarantees the existence of a constant $\delta=\delta(2^{-\beta})$ such that
\begin{equation}
\sup_{B_{1/2}}v_1(x)\leq 2^{-\beta},
\end{equation}
whenever $0<\gamma\leq\delta(2^{-\beta})$. Now, let us define
$$
v_2(x)=2^{\beta}v_1(x/2)\quad\text{in } B_1.
$$
Since $v_1$ solves the equation \eqref{Eq.v1}, we can see that $v_2$ is a viscosity solution of
$$
|\nabla v|^pF_2(D^2v)+\mathfrak{a}_2(x)|\nabla v|^q=\lambda_2(x)v^{\mu}_+,
$$
where
\begin{align*}
F_2(M) & \coloneqq  2^{\beta-2}F_1(2^{2-\beta}M),\\[0.2cm]
\mathfrak{a}_2(x)& \coloneqq 2^{\beta(p+1-q)-p-2+q}\mathfrak{a}_1(x/2),\\[0.2cm]
\lambda_2(x) & \coloneqq 2^{\beta(p+1-\mu)-p-2}\lambda_1(x/2).
\end{align*}
Once again, $F_2$ is also an $(\lambda,\Lambda)$-elliptic operator, and in addition by the choice of $\beta$, we have
$$
\beta(p+1-q)-p-2+q\leq 0\quad\text{and}\quad \beta(p+1-\mu)-p-2\leq 0.
$$
From this, we deduce 
$$
\|\mathfrak{a}_2(x)\|_{L^{\infty}(B_1)}\leq\|\mathfrak{a}_1(x)\|_{L^{\infty}(B_1)}\leq\frac{\gamma}{2}\quad\text{and}\quad\lambda_2\leq\lambda_1\leq\frac{\gamma}{2},
$$
which ensures that $v_2$ is under the hypothesis of the Lemma \ref{estimaHA}, that is
\begin{equation}
\sup_{B_{1/2}}v_2(x)\leq 2^{-\beta}.
\end{equation}
Next, arguing inductively, for each $k\in\mathbb{N}$ define
$$
v_k(x)\coloneqq 2^{\beta}v_{k-1}(x/2).
$$
Using the same arguments above, we invoke the Lemma \ref{estimaHA} once again, which implies
$$
\sup_{B_{1/2}}v_k(x)\leq 2^{-\beta}.
$$
After a rescaling
\begin{equation}
\sup_{B_{2^{-k}}}v_{1}\leq 2^{-k\beta}.
\label{v1estimate}
\end{equation}
Now, fix $0<r\leq\frac{\rho}{2}$, and take $k\in\mathbb{N}$ such that
\begin{equation}\label{lastestimate}
2^{-(k+1)}<\frac{r}{\rho}\leq 2^{-k}.
\end{equation}    
So, we have
\begin{equation*}
\sup_{B_r(x_0)}u(x)\leq\sup_{B_{\rho2^{-k}}(x_0)}u(x)=\tau\sup_{B_{2^{-k}}}v_1(x).
\end{equation*}
Consequently, by \eqref{v1estimate} and \eqref{lastestimate} we conclude that
$$
\sup_{B_r(x_0)}u(x) \leq  \tau 2^{-k\beta}\leq  (\tau\, 2^{\beta})2^{-(k+1)\beta} \leq \left(\frac{2}{\rho}\right)^{\beta}\tau\cdot r^{\beta},
$$
which prove the theorem.
\end{proof}

As a consequence of the Theorem \ref{improvement_regularity} and with the support of a barrier argument, we have the following result of positivity of solutions under certain conditions, which can also be seen as a strong maximum principle.

\begin{corollary}\label{strong.maximum.principle}
Let $u$ be a non-negative, bounded viscosity solution to \eqref{equation} with $\mu=p+1$ and $\mathfrak{a}\geq 0$ in $\Omega$. Then, if $\max\{0,p\}<q<p+1$, the following dichotomy
holds: either $u>0$ or $u\equiv 0$ in $\Omega$.
\end{corollary}
\begin{proof}
    Following the ideas in \cite{Leitao}, we shall do the proof by contradiction. To do this, suppose that there exists $x'\in\Omega$ such that $u(x')=0$, but $u\not\equiv 0$ in $\Omega$. Let $x_0\in\Omega$ be such that $u(x_0)>0$, then we can suppose without loss of generality that
    $$d_0\coloneqq\text{dist}(x_0,\partial\{u>0\})<\frac{1}{5}\text{dist}(x_0,\partial\Omega).$$
    Let $\eta$ and $s$ be positive values which will be chosen later, and define the following barrier function
    $$\theta_{\eta,s}(x)\coloneqq\eta\frac{e^{-s\frac{|x-x_0|^2}{d^2_0}}-e^{-s}}{e^{-\frac{s}{4}}-e^{-s}}.$$
    For appropriate values of $\eta$ and $s$, the aim is to show that $\theta_{\eta,s}$ solves in viscosity sense
\begin{equation*}
\left\{
\begin{array}{rclll}
\mathcal{L}[\theta_{\eta,s}](x) &\geq & 0 & \text{in}  & B_{d_0}(x_0)\setminus B_{\frac{d_0}{2}}(x_0), \\
\theta_{\eta,s} & = & \eta & \text{in} & \partial B_{\frac{d_0}{2}}(x_0),\\
\theta_{\eta,s} & = & 0 & \text{in} & \partial B_{d_0}(x_0),
\end{array}
\right.
\end{equation*}
where 
$$
\mathcal{L}[u](x)=|\nabla u|^p F(D^2u)+\mathfrak{a}(x)|\nabla u|^q-\lambda_0(x)u^{p+1}_+(x).
$$
Let $\theta(x)=\theta_{\eta,s}(x)$ be, then a direct computation shows us
\begin{align}
\theta_i(x)&=-\frac{\eta s}{d^2_0}\cdot\frac{e^{-s\frac{|x-x_0|^2}{d^2_0}}}{e^{-\frac{s}{4}}-e^{-s}}\cdot 2(x_i-(x_0)_i),\label{PosiGrad}\\
\theta_{ij}(x) &= \frac{2\eta s}{d^2_0}\cdot\frac{e^{-s\frac{|x-x_0|^2}{d^2_0}}}{e^{-\frac{s}{4}}-e^{-s}}\left[-\delta_{ij}+\frac{2s}{d^2_0}(x_i-(x_0)_i)\cdot(x_j-(x_0)_j)\right].
\end{align}
Note that, for $x\in B_{d_0}\setminus B_{\frac{d_0}{2}}$ and $\hat{x}=(|x|+(x_0)_1,(x_0)_2,...,(x_0)_n)$, we find
\begin{equation*}
\begin{array}{lllll}
\theta_{11}(\hat{x}) &=& \frac{2\eta s}{d^2_0}\cdot\frac{e^{-s\frac{|x-x_0|^2}{d^2_0}}}{e^{-\frac{s}{4}}-e^{-s}}\left[-1+\frac{2s}{d^2_0}|x|^2\right], &\\[0.5cm]
\theta_{ii}(\hat{x}) &=& -\frac{2\eta s}{d^2_0}\cdot\frac{e^{-s\frac{|x-x_0|^2}{d^2_0}}}{e^{-\frac{s}{4}}-e^{-s}}, & \text{if} & i\neq 1;\\[0.5cm]
\theta_{ij}(\hat{x}) &=& 0, & \text{if} & i\neq j.
\end{array}
\end{equation*}
If $s>2$, we get $\theta(\hat{x})_{11}>0$, because $|x|^2\geq \frac{d^2_0}{4}$. Thus, by symmetric invariance of $\theta$ and ellipticity of $F$, we have \begin{align}\label{OperatorEli}
F(D^2\theta(x)) &\geq \mathcal{P}^-_{\lambda,\Lambda}((D^2\theta(x))^+)+F(D^2\theta(x)^-)\nonumber\\
&\geq  \mathcal{P}^-_{\lambda,\Lambda}((D^2\theta(x))^+)-\Lambda\|(-D^2\theta(x))^+\|\nonumber\\
&\geq  \lambda\frac{2\eta s}{d^2_0}\cdot\frac{e^{-s\frac{|x-x_0|^2}{d^2_0}}}{e^{-\frac{s}{4}}-e^{-s}}\left[-1+\frac{s}{2}\right]-\Lambda\frac{2\eta s}{d^2_0}\cdot\frac{e^{-s\frac{|x-x_0|^2}{d^2_0}}}{e^{-\frac{s}{4}}-e^{-s}}\nonumber\\
&= \frac{2\eta s}{d^2_0}\cdot\frac{e^{-s\frac{|x-x_0|^2}{d^2_0}}}{e^{-\frac{s}{4}}-e^{-s}}\left[-\lambda+\frac{\lambda s}{2}-\Lambda\right].
\end{align}
Hence, by \eqref{PosiGrad} and \eqref{OperatorEli}
\begin{eqnarray*}
&&|D\theta|^p F(D^2\theta)+\mathfrak{a}(x)|D\theta|^q-\lambda_0(x)\theta^{p+1}_+(x)\\
&\geq & |x-x_0|^p \left[2\eta s\frac{e^{-s\frac{|x-x_0|^2}{d^2_0}}}{d^2_0(e^{-\frac{s}{4}}-e^{-s})}\right]^{p+1}\left[-\lambda+\frac{\lambda s}{2}-\Lambda\right]\\
&& +\mathfrak{a}(x)\left[2\eta s\frac{e^{-s\frac{|x-x_0|^2}{d^2_0}}}{d^2_0(e^{-\frac{s}{4}}-e^{-s})}\right]^q-\|\lambda_0\|_{\infty}\left[\eta\frac{e^{-s\frac{|x-x_0|^2}{d^2_0}}-e^{-s}}{e^{-\frac{s}{4}}-e^{-s}}\right]^{p+1}\\
&\geq & \mathcal{A}^p \left[2\eta s\frac{e^{-s\frac{|x-x_0|^2}{d^2_0}}}{d^2_0(e^{-\frac{s}{4}}-e^{-s})}\right]^{p+1}\left[-\lambda+\frac{\lambda s}{2}-\Lambda\right]\\
&& -\|\lambda_0\|_{\infty}\left[\eta\frac{e^{-s\frac{|x-x_0|^2}{d^2_0}}-e^{-s}}{e^{-\frac{s}{4}}-e^{-s}}\right]^{p+1}\\
&\geq & \left[\mathcal{A}^p2^{p+1}s^{p+1}\left(-\lambda+\frac{\lambda s}{2}-\Lambda\right)-\|\lambda_0\|_{\infty}\right]\left[\eta\frac{e^{-s\frac{|x-x_0|^2}{d^2_0}}-e^{-s}}{e^{-\frac{s}{4}}-e^{-s}}\right]^{p+1},
\end{eqnarray*}
where $\mathcal{A}^p=\min\{d^p_0,(d_0/2)^p\}$. Then, for $s$ large enough such that 
$$
s^{p+1}\left(-\lambda+\frac{\lambda s}{2}-\Lambda\right)\geq \frac{\|\lambda_0\|_{\infty}}{\mathcal{A}^p2^{p+1}},
$$
we conclude that
\begin{equation}\label{inequaliteCon}
\mathcal{L}[\theta_{\eta,s}](x)\geq 0\quad\text{in}\quad B_{d_0}(x_0)\setminus B_{\frac{d_0}{2}}(x_0).
\end{equation}    
From the above estimates, we see that inequality \eqref{inequaliteCon} holds for all $\eta>0$. With that, by taking  $\eta$ small enough such that $\eta\leq \inf_{B_{\frac{d_0}{2}}(x_0)}u(x)$, we have in the viscosity sense
\begin{equation*}
\begin{cases}
        \mathcal{L}[u](x)= 0\leq\mathcal{L}[\theta_{\eta,s}](x)\quad\text{in}\quad B_{d_0}(x_0)\setminus B_{\frac{d_0}{2}}(x_0),\\
        \quad\quad\quad\quad\quad\quad\,\,\,\,\,\,\theta_{\eta,s}\leq u\quad\text{on}\quad\partial B_{d_0}(x_0)\cup \partial B_{\frac{d_0}{2}}(x_0).
    \end{cases}
    \end{equation*}
    Therefore, by comparison principle, Lemma \ref{ComparisonPrin}
    \begin{equation}\label{ComparisonIneq}
        \theta_{\eta, s}\leq u\quad\text{in } B_{d_0}(x_0)\setminus B_{\frac{d_0}{2}}(x_0).
    \end{equation}

Now, notice that
\begin{align}
\inf_{B_{d_0}(x_0)\setminus B_{\frac{d_0}{2}}(x_0)}|D\theta_{\eta,s}(x)|& \geq \frac{\eta s d^{-1}_0 e^{-s}}{e^{-\frac{s}{4}}-e^{-s}}\nonumber\\[0.2cm]
&\geq \frac{5\eta se^{-s}}{\operatorname{dist}(x_0,\partial\Omega)(e^{-\frac{s}{4}}-e^{-s})}\nonumber\\[0.2cm]
&\coloneqq \varsigma>0. \label{GradEst}
\end{align}
Moreover, for any $\max\{p,0\}<\tilde{\mu}<p+1$ fixed, define
$$
\omega(x)\coloneqq \lambda_0(x)u^{p+1-\tilde{\mu}}_+(x).
$$
Then $u$ solves the following equation
$$
|\nabla u|^pF(D^2u)+\mathfrak{a}(x)|\nabla u|^q=\omega(x)u^{\tilde{\mu}}_+(x)\quad\text{in } \Omega.
$$
Thus, for $z\in\partial B_{d_0}(x_0)\cap\partial\{u>0\}$ we can apply Theorem \ref{growth} which gives us
$$
\sup_{B_r(z)}u(x)\leq \mathrm{C}(n,\lambda,\Lambda,p,q,\mu,\|\mathfrak{a}\|_{\infty},\|\omega\|_{\infty})\cdot r^{\tilde{\beta}},
$$
with $\tilde{\beta}=\min\left\{\frac{p+2-q}{p+1-q},\frac{p+2}{p+1-\tilde{\mu}}\right\}$. Assuming that $\tilde{\beta}=\frac{p+2-q}{p+1-q}$, we obtain 
$$
\sup_{B_r(z)}u(x)\leq Cr^{\frac{p+2-q}{p+1-q}}\leq C r^{p+2-q},
$$
for $r\ll 1$, since $p+1-q<1$. Now, note that $p+2-q>1$, then we can choose $0<r_0\ll 1$ such that
\begin{equation}\label{contradii}
Cr^{p+2-q}_0\leq\frac{1}{2}\varsigma\, r_0.
\end{equation}  
Finally, by \eqref{ComparisonIneq}, \eqref{GradEst} and \eqref{contradii}, we conclude
    $$\varsigma\,r_0\leq\sup_{B_{r_0}(z)}|\theta(x)-\theta(z)|\leq\sup_{B_{r_0}(z)}\theta(x)\leq\sup_{B_{r_0}(z)}u\leq Cr^{p+2-q}_0\leq\frac{1}{2}\varsigma\,r_0,$$
    which is a contradiction. Similarly, if $\tilde{\beta}=\frac{p+2}{p+1-\tilde{\mu}}$  we also reach a contradiction and the proof of Corollary is complete.
\end{proof}

\section{Non-degeneracy estimates: Proof of Theorem \ref{ND_estimates}}\label{ND_Section}

Next, we prove that the growth rate on the free boundary in the Theorem \ref{improvement_regularity} is indeed optimal by imposing a restriction on the value of $\beta$. This is done through downward controlled growth, also of order $\beta$.

\begin{proof}[\textit{Proof of Theorem \ref{ND_estimates}}]
Notice that, by assumption $\beta=\frac{p+2}{p+1-\mu}$. In the sequel, due to the continuity of solutions, it is sufficient to prove that such an estimate is satisfied just at point within $\{u>0\}\cap B_{1/2}$.

Without loss of generality, we assume that $0\in \{u>0\}\cap B_{1/2}$ and define the scaled function
$$
    u_r(x)=\frac{u(rx)}{r^\beta}\quad\text{for}\quad x\in B_1.
$$
Also, let us introduce the comparison function
$$
\Xi(x)= A|x|^\beta,
$$
where $A$ is a constant to be chosen later. A simple calculation shows us that
$$
    \Xi_{ij}(x)= A\beta\left[(\beta-2)|x|^{\beta-4}x_ix_j +|x|^{\beta-2}\delta_{ij}\right].
$$
Then in point of the form $(|x|,0,\ldots,0)$, we obtain
\begin{equation*}
    \begin{array}{lcr}
    \Xi_{11} =  A\beta(\beta-1)|x|^{\beta-2},& & \\[0.2cm]
    \Xi_{ii} = A\beta|x|^{\beta-2} &\text{if} & i>1,\\[0.2cm]
    \Xi_{ij} = 0  &\mbox{if} & i\neq j.
    \end{array}
\end{equation*}

Now, consider
\begin{align*}
    \mathcal{F}_r(x,M) & \coloneqq r^{2-\beta}F\left(rx,r^{\beta -2}M\right),\\[0.2cm]
\mathfrak{a}_r(x) &\coloneqq  r^{p+2-q-\beta(p+1-q)}\mathfrak{a}(rx),\\[0.2cm]
\lambda_{0,r}(x) &\coloneqq \lambda_0(rx).
\end{align*}
In particular, by symmetric invariance of $\Xi$ and \textbf{\ref{unif ellipticity}}, we obtain that 
\begin{align}
\mathcal{F}_r(D^2\Xi(x)) &=r^{2-\beta}F\left(r^{\beta -2}D^2\Xi(x)\right)\nonumber\\[0.2cm]
& \leq \mathcal{P}^+_{\lambda,\Lambda}(D^2\Xi(x))\nonumber\nonumber\\[0.2cm]
& \leq \Lambda[A\beta(\beta-1)+(n-1)A\beta]|x|^{\beta-2}\nonumber\\[0.2cm]
&= \Lambda A\beta (n+\beta-2)|x|^{\beta-2}\label{ND1}.
\end{align}
Then, by \eqref{ND1}, we obtain
\begin{eqnarray}
& & |\nabla \Xi|^p\mathcal{F}_r(D^2\Xi)+\mathfrak{a}_r(x)|\nabla \Xi|^q-\lambda_{0,r}(x)\,\Xi^{\mu}\nonumber\\[0.2cm]
    &\leq & \Lambda A^{p+1}\beta^{p+1}(\beta+n-2)|x|^{(\beta-1)(p+1)-1}\label{NDEqInequality}\\
    & & +\|\mathfrak{a}\|_{\infty}r^{p+2-q-\beta(p+1-q)}A^q\beta^q|x|^{(\beta-1) q}-(\inf_{B_1}\lambda_0)A^{\mu}|x|^{\beta \mu}.\nonumber
\end{eqnarray}
Note that
\begin{equation}\label{NDbeta1}
    (\beta-1)p+\beta-2 \leq (\beta-1)q \Leftrightarrow \beta\leq\frac{p+2-q}{p+1-q},
\end{equation}
and
\begin{equation}\label{NDbeta2}
    (\beta-1)p+\beta-2 = \beta\mu \Leftrightarrow \beta=\frac{p+2}{p+1-\mu}.
\end{equation}
Hence, by \eqref{NDEqInequality}, \eqref{NDbeta1}, and \eqref{NDbeta2}, we have 
\begin{eqnarray}
    & & |\nabla \Xi|^p\mathcal{F}_r(x,D^2\Xi)+\mathfrak{a}_r(x)|\nabla \Xi|^q-\lambda_{0,r}(x)\,\Xi^{\mu}\nonumber\\[0.2cm]
    &\leq & \left[\Lambda A^{p+1}\beta^{p+1}(\beta+n-2)
    +\|\mathfrak{a}\|_{\infty}A^q\beta^q - (\inf_{B_1}\lambda_0)A^{\mu}\right]|x|^{\beta\mu}.\label{NDlastinequality}
\end{eqnarray}
By hypothesis, 
$$
    \frac{p+2-q}{p+1-q}\geq\frac{p+2}{p+1-\mu},
$$ 
which implies $\mu<q$, and furthermore $\mu<p+1$, then for $0<A\ll 1$, the term in \eqref{NDlastinequality} is non-positive. In other words, we conclude that
$$
|\nabla \Xi|^p\mathcal{F}_r(D^2\Xi)+\mathfrak{a}_r(x)|\nabla \Xi|^q\leq\lambda_{0,r}(x)\Xi^\mu\quad\text{in } B_1,
$$
in the viscosity sense. Moreover,
$$
|\nabla u_r|^p\mathcal{F}_r(D^2u_r)+\mathfrak{a}_r(x)|\nabla u_r|^q=\lambda_{0,r}(x)u_r^\mu.
$$
Finally, there exists a point $z_0\in \partial B _1$ such that
\begin{equation}\label{NC}
     u_r(z_0)>\Xi(z_0),
\end{equation}
otherwise, if $u_r\leq \Xi$ on the whole boundary of $B_1$, the Comparison Principle (Lemma \ref{ComparisonPrin}), would imply that $u_r\leq \Xi$ in $B_1$, which clearly contradicts the assumption that $u_r(0)>0$. Finally, by \eqref{NC}, we obtain
$$
    \sup_{B_1}u_r\geq\sup_{\partial B_1}u_r \geq u_r(z_0)>\Xi(z_0) =A.
$$
Thus, by definition of $u_r$, we finish the proof of the theorem.
\end{proof}

To complete the section, we display a counter example showing that non-degeneracy can fails when
$$
    \frac{p+2-q}{p+1-q}<\frac{p+2}{p+1-\mu}.
$$

\begin{example}
In general, when $\beta=(p+2-q)/(p+1-q)$ non-degeneracy can fail even considering the exponent
$$
\overline{\beta}=\max\left\{\frac{p+2-q}{p+1-q},\frac{p+2}{p+1-\mu}\right\}.
$$
Indeed, given $\varepsilon>0$, the function $u(x)=|x|^{\frac{p+2}{p+1-\mu}+\varepsilon}$ satisfies
$$
|\nabla u|^p\Delta u+\mathfrak{a}(x)|\nabla u|^{\mu}=\lambda_0(x)u^{\mu}\quad\text{in } B_1,
$$
where
\begin{equation*}
\begin{array}{l}
    \mathfrak{a}(x)=|x|^{\mu+\varepsilon(p+1-\mu)}+c_0|x|^{\mu},\\
    \lambda_0(x)=\overline{c}|x|^{\varepsilon(p+1-\mu)}+\gamma,
\end{array}
\end{equation*}
for $\gamma>0$, and
\begin{eqnarray*}
    \overline{c}\ &=& \left(\frac{p+2}{p+1-\mu}+\varepsilon+n-2\right)\left(\frac{p+2}{p+1-\mu}+\varepsilon\right)^{p+1}+\left(\frac{p+2}{p+1-\mu}+\varepsilon\right)^{\mu}\\
    c_0 &=& \gamma\left(\frac{p+2}{p+1-\mu}+\varepsilon\right)^{-\mu}.
\end{eqnarray*}
In this case, $q=\mu$ and then
$$
\frac{p+2-q}{p+1-q}<\frac{p+2}{p+1-\mu},$$
that is, $\overline{\beta}=(p+2)/(p+1-\mu)$, but there is no constant $C>0$ such that
$$
Cr^{\overline{\beta}}\leq\sup_{B_r}u=r^{\overline{\beta}+\varepsilon},
$$
for all $r\in(0,1/2)$.
\end{example}

\subsection*{Acknowledgments}

Rafael R. Costa has been supported by the Coordenação de Aperfeiçoamento de Pessoal de Nível Superior (CAPES)–Brasil under Grant No. 88887.702383/2022-00. Ginaldo S. S\'{a} expresses gratitude for the PDJ-CNPq-Brazil, Postdoctoral Fellowship Grant No. 174130/2023-6.

\end{document}